\documentclass[12pt]{amsart}
\usepackage{amssymb, amscd}
\usepackage[firstinits=true]{biblatex}
\usepackage{hyperref}
\usepackage{xcolor}\hypersetup{linktocpage,colorlinks=true,linkcolor=black,citecolor=black}
\usepackage{mathrsfs}
\usepackage{centernot}
\usepackage{lipsum}
\usepackage[margin=1.4in]{geometry}

\renewbibmacro*{in:}{}

\numberwithin{equation}{section}

\bibliography{refs}

\theoremstyle{plain}
\newtheorem{theorem}{Theorem}[section]
\newtheorem{lemma}[theorem]{Lemma}
\newtheorem{corollary}[theorem]{Corollary}
\newtheorem{proposition}[theorem]{Proposition}

\theoremstyle{definition}
\newtheorem{definition}[theorem]{Definition}
\newtheorem{example}[theorem]{Example}
\newtheorem{remark}[theorem]{Remark}

\theoremstyle{plain}
\newtheorem*{theorem1}{Theorem 1.1}

\theoremstyle{definition}
\newtheorem*{notation}{Notation}

\newcommand{\coker}{\mathrm{coker}}
\newcommand{\im}{\mathrm{im}}
\newcommand{\ind}{\mathrm{ind}}
\newcommand{\tr}{\mathrm{tr}}

\begin{document}

\title{Joint torsion equals the determinant invariant}

\author{Joseph Migler}

\address{Department of Mathematics \\ University of Colorado \\ Boulder, CO 80309-0395}

\email{joseph.migler@colorado.edu}

\date{}

\begin{abstract}
A determinant in algebraic $K$-theory is associated to any two almost commuting Fredholm operators.
On the other hand, one can calculate a homologically defined invariant known as joint torsion.  We answer in the affirmative a conjecture of Richard Carey and Joel Pincus, namely that these two invariants agree.  In particular, this implies that joint torsion is norm continuous, depends only on the images of the operators modulo trace class, and satisfies the expected Steinberg relations.  Moreover, we show that the determinant invariant of two commuting operators can be computed simply as a determinant on a finite dimensional vector space.
\\
\end{abstract}

\subjclass[2000]{Primary 19C20; Secondary 47A13, 47B35, 18G35}

\keywords{joint torsion; perturbation vectors; determinants in $K$-theory; multivariable operator theory}


\maketitle

\tableofcontents

\newpage

\section{Introduction}

Let $A$ and $B$ be two invertible operators on a Hilbert space $\mathscr H$ that commute modulo the trace class $\mathcal L^1 (\mathscr H)$.  The multiplicative commutator $ABA^{-1}B^{-1}$ is an invertible determinant class operator, and therefore has a nonzero determinant.  The assignment 
$
(A,B) \mapsto \det ABA^{-1}B^{-1}
$
is bimultiplicative and skew-symmetric.  Moreover, $\det ABA^{-1}B^{-1}= \det \tilde A \tilde B \tilde A^{-1} \tilde B^{-1}$ for any invertible trace class perturbations $\tilde A$ and $\tilde B$ of $A$ and $B$, respectively.

L. Brown observed in \cite{Brown} that this is a special case of a more general phenomenon.  Indeed, any two bounded Fredholm operators $A$ and $B$ that commute modulo trace class have invertible commuting images $a$ and $b$ in the quotient $\mathcal L / \mathcal L^1$.  Here, $\mathcal L = \mathcal L ( \mathscr H )$ is the algebra of bounded linear operators on $\mathscr H$, and $\mathcal L^1 = \mathcal L^1( \mathscr H )$ is the ideal of trace class operators.  Then there is a Steinberg symbol $ \{ a,b \} $ in the second algebraic $K$-group $K_2(\mathcal L / \mathcal L^1)$, and thus a determinant invariant $d(a,b) = \det \partial \{ a,b \} \in \mathbf C^\times$.  See Section \ref{subsec:preliminaries-K-theory} below for more details.  Brown showed that when $A$ and $B$ are invertible, $d(a,b) = \det ABA^{-1}B^{-1}$.  This immediately yields the remarkable fact that the determinant of the multiplicative commutator depends only on the Steinberg symbol in $K$-theory.  See also the paper of J. Helton and R. Howe \cite{Helton-Howe1}.

As an example, consider two nonvanishing smooth functions $f$ and $g$ on the unit circle.  The Toeplitz operators $T_f$ and $T_g$ are Fredholm and commute with each other modulo trace class.
In \cite[Theorem 7]{Joint}, R. Carey and J. Pincus use a calculation of A. Beilinson \cite{Beilinson} to show that
\begin{equation*}
d(T_f + \mathcal L^1, T_g + \mathcal L^1) = \exp \frac{1}{2\pi i} \left( \int_{S^1} \, \log f \, d(\log g) - \log g(p) \int_{S^1} \, d(\log f) \right).
\end{equation*}
The integrals are taken counterclockwise starting at any point $p \in S^1$.  For a continuous function $h$, the logarithm $\log h$ is interpreted as continuous in the argument $\theta$.  For example, if $z = e^{i \theta}$, then $\log (z^2) = 2 i \theta$ for $0 \leq \theta < 2\pi$.

More recently, J. Kaad has shown in \cite{Kaad-Comparison} that the determinant invariant coincides with the Connes-Karoubi multiplicative character \cite{Connes-Karoubi}.  Furthermore, Kaad has constructed a product in relative $K$-theory and investigated the relative Chern character with values in continuous cyclic homology \cite{Kaad-Calculation}.  He uses these results to calculate the multiplicative character applied to Loday products of exponentials.

Let $A$ and $B$ be commuting operators with images $a$ and $b$ in $\mathcal L / \mathcal L^1$.  Carey and Pincus showed in \cite[Theorem 1.1]{Reciprocity} that
\begin{equation} \label{eq:Lefschetz-z1-z2}
d(a-z_1,b-z_2) =  \frac{\det (B-z_2)|_{\ker (A-z_1)}}{\det (B-z_2)|_{\coker \, (A-z_1)}} \, \frac{\det (A-z_1)|_{\coker \, (B-z_2)}}{\det (A-z_1)|_{\ker (B-z_2)}}
\end{equation}
whenever $a-z_1$ and $a-z_2$ are invertible and $(z_1,z_2) \notin \sigma_T(A,B)$, the Taylor joint spectrum of $A$ and $B$.  However, the left hand side of \eqref{eq:Lefschetz-z1-z2} is defined even if $(z_1,z_2) \in \sigma_T(A,B)$.  Thus, they constructed an invariant $\tau(A,B)$, known as joint torsion, for any two commuting Fredholm operators $A$ and $B$.  This generalizes the multiplicative Lefschetz number on the right hand side of \eqref{eq:Lefschetz-z1-z2}.

To define joint torsion, Carey and Pincus use the notion of algebraic torsion.  For any exact sequence $(V_\bullet, d_\bullet)$ of finite dimensional vector spaces, there is a canonical generator $\tau(V_\bullet, d_\bullet)$ of the determinant line $\det V_\bullet$, known as the torsion.  For example, the torsion of an isomorphism of a finite dimensional vector space is its determinant.  
As another example, J.-M. Bismut, H. Gillet, and C. Soul\'{e} have shown that the Ray-Singer analytic torsion can be calculated as the norm of $\tau(V_\bullet, d_\bullet)$ \cite[Proposition 1.5]{BGS1}.
See also Section \ref{subsec:preliminaries-jt} below.

Kaad has generalized the notion of joint torsion to $n\geq 2$ commuting operators \cite{Kaad}.  Moreover, he shows that joint torsion is multiplicative, satisfies cocycle identities, and is trivial under appropriate Fredholm assumptions.  He has also investigated the relationship between joint torsion on the one hand, and determinant functors and $K$-theory of triangulated categories on the other.

Carey and Pincus extended their definition of joint torsion to two almost commuting Fredholm operators \cite{Perturbation}.  Thus, let $A$ and $B$ be Fredholm operators with $[A,B] \in \mathcal L^1$, and moreover, assume the existence of operators $C$ and $D$ such that $AB=CD$, $A-D \in \mathcal L^1$, and $B-C \in \mathcal L^1$.  They then proceed as before, defining $\tau(A,B,C,D)$ in terms of short exact sequences of Koszul complexes.  However, the result is no longer a scalar, but rather an element of a certain determinant line.  To obtain a scalar, Carey and Pincus associate a perturbation vector $\sigma_{A,A'}$ to each pair of Fredholm operators $A$ and $A'$ with $A-A' \in \mathcal L^1$.  The perturbation vector $\sigma_{A,A'}$ is a canonical generator of the determinant line $\det H_\bullet(A) \otimes \det H_\bullet(A')^*$.
Here, we use the notation $\det H_\bullet(A) = ( \det \ker A)^* \otimes ( \det \coker \, A)$, where $\det$ denotes the top exterior power.  See also Section \ref{subsec:preliminaries-almost} below for more details on the constructions.

Perturbation vectors can be seen as a generalization of the classical perturbation determinant $\det A^{-1} A'$.  Carey and Pincus have shown in \cite[Theorem 15]{Perturbation} that perturbation vectors form a nonvanishing section of a Quillen determinant line bundle \cite{Quillen}.  Moreover, they have applied joint torsion to Toeplitz operators, especially problems where standard techniques only apply to symbols with zero winding number.  They prove Szeg\H{o} limit theorems on the asymptotic behavior of determinants of Toeplitz operators whose symbols have nonzero winding number \cite{Perturbation}.  In the case when $f,g \in H^{\infty}(S^1)$, the joint torsion $\tau(T_f,T_g)$ is the product of tame symbols \cite{Joint}, and can be expressed in terms of Deligne cohomology \cite{Deligne}.  In particular, the determinant invariant $d(T_f + \mathcal L^1, T_g + \mathcal L^1)$ is equal to the joint torsion $\tau(T_f, T_g)$ when $f$ and $g$ are smooth functions in $H^{\infty}(S^1)$.

More generally, Carey and Pincus state in \cite[Section 8, p.~345]{Perturbation}:
\begin{quote}

The existence of the identification map (in the existence theorem for the perturbation vector) has uncovered a basic problem - which deserves to be stated as a question or perhaps as a conjecture:

Let $a, b$ be commuting units in $\mathcal L(\mathscr H) / \mathcal L^1(\mathscr H)$.  Let $A,B,C,D$ be elements in $\mathcal L(\mathscr H)$ so that $AB=CD$ and let $A,D$ denote lifts of $a$ and $B,C$ denote lifts of $b$.  In what generality is it true that $\det \partial \{ a,b \} = \tau(A,B,C,D)$?
\end{quote}

This is true for invertible operators $A$, $B$, $C$, and $D$ \cite[Section 8]{Perturbation}.
Carey and Pincus proved in \cite[Theorem 1.1]{Reciprocity} that this is also true for commuting Fredholm operators $A$ and $B$ with acyclic Koszul complex $K_\bullet(A,B)$.  
More generally, they have shown in unpublished work that $d(a,b) = \tau(A,B,B,A)$ for any commuting Fredholm operators $A$ and $B$ (see \cite[Theorem 2]{Joint}).

The main result of the present paper is the following theorem, which answers the above question in full generality.

\begin{theorem} \label{thm:jt=di}
Let $a$ and $b$ be commuting units in $\mathcal L(\mathscr H) / \mathcal L^1(\mathscr H)$.
Let $A, D \in \mathcal L(\mathscr H)$ be lifts of $a$, and let $B, C \in \mathcal L(\mathscr H)$ be lifts of $b$ such that $AB=CD$.
Then $d(a,b) = \tau(A,B,C,D)$.
\end{theorem}

The organization of this paper is as follows.  In Section \ref{sec:preliminaries}, we review the determinant invariant in algebraic $K$-theory and joint torsion, first for commuting operators, and then for almost commuting operators.  In Section \ref{sec:finite-dimensional}, we show that joint torsion is trivial in a finite dimensional space.  In the final section, we establish a number of key results on factoring perturbation vectors and joint torsion.  We then prove Theorem \ref{thm:jt=di} and discuss a number of consequences.

\subsection*{Acknowledgments}
I am indebted to Alexander Gorokhovsky for introducing me to joint torsion and for countless enlightening conversations.  This work would not have happened without his generosity in sharing both his time and knowledge.
I would also like to thank Richard Carey for his patience in answering my questions.
I am grateful to Nigel Higson for numerous enlightening conversations.  
I wish to thank Jerry Kaminker for encouraging me to investigate joint torsion, Deligne cohomology, and the determinant invariant.
I would like to thank Matthias Lesch for directing me to the literature on quasicomplexes.
I would also like to thank Mariusz Wodzicki for an interesting conversation on the subject.

\section{Preliminaries} \label{sec:preliminaries}

\subsection{The determinant invariant} \label{subsec:preliminaries-K-theory}

For any unital ring $R$ and ideal $I$, there are algebraic $K$-groups $K_i(R)$, $K_i(R/I)$, and $K_i(R,I)$ that fit into Quillen's long exact sequence
\[
\dots \to K_{i+1}(R/I) \xrightarrow{\partial} K_i(R,I) \to K_i(R) \to K_i(R/I) \xrightarrow{\partial} \dots
\]
We are mainly interested in the map $\partial : K_{2}(R/I) \xrightarrow{} K_1(R,I)$, so let us begin by recalling the relevant definitions.  See also \cite[Chapters 2 and 4]{Rosenberg}.

Let $GL(R,I)$ denote the kernel of the group homomorphism $GL(R) \to GL(R/I)$ induced by the quotient map $R \to R/I$.  Denote by $E(R,I)$ the subgroup of elementary matrices generated by matrices that differ from the identity by at most one off-diagonal element of $I$.  Let $[GL(R), E(R,I)] \subseteq GL(R,I)$ denote the subgroup generated by all elements of the form
\[
ghg^{-1}h^{-1}, \quad g\in GL(R), h\in E(R,I).
\]
\begin{definition}
$K_1(R,I) = GL(R,I) / [GL(R), E(R,I)]$.
\end{definition}
In the case when $R$ is the ring $\mathcal L$ of bounded linear operators on a Hilbert space and $I$ is the ideal $\mathcal L^1$ of trace class operators, the Fredholm determinant induces a surjective group homomorphism
\[
\det : K_1(\mathcal L, \mathcal L^1) \to \mathbf C^\times.
\]
In fact, $K_1(\mathcal L, \mathcal L^1) = V \oplus C^\times$, where $V$ is the additive group of a vector space of uncountable linear dimension, and $\det$ is projection onto the second factor \cite{Anderson-Vaserstein}.

Now we describe Milnor's second algebraic $K$-group.  The Steinberg group $St_n(R)$ is the group with generators $x_{ij}(a)$ for $a \in R$, $i\neq j$, $1 \leq i,j \leq n$, and relations
\begin{align*}
& x_{ij}(a) x_{ij}(b) = x_{ij}(a+b) & \\
& x_{ij}(a) x_{kl}(b) x_{ij}(a)^{-1} x_{kl}(b)^{-1} = 1, & j \neq k, i \neq l\\
& x_{ij}(a) x_{jk}(b) x_{ij}(a)^{-1} x_{jk}(b)^{-1} = x_{ik}(ab), & i,j,k \, \text{distinct}
\end{align*}
Since the generators of the group of elementary matrices $E_n(R)$ satisfy these same relations, we have a map $St_n(R) \to E_n(R)$, and hence a natural map $\phi$ on the inductive limits $St(R)$ and $E(R).$

\begin{definition}
$K_2(R) = \ker (\phi : St(R) \to E(R))$.
\end{definition}

In fact, $St(R)$ is the universal central extension of $E(R)$, and $K_2(R)$ is isomorphic to the second homology group $H_2(E(R), \mathbf Z)$ \cite[Theorem 4.2.7 and Corollary 4.2.10]{Rosenberg}.

Next we define specific elements in $K_2(R)$ known as Steinberg symbols, which turn out to be quite useful.  For example, $K_2$ of a field is generated by its Steinberg symbols.  For any $a \in R^\times$ and $i \neq j$, define elements $w_{ij}(a), h_{ij}(a) \in St(R)$ by
\[
w_{ij}(a) = x_{ij}(a) x_{ji}(-a^{-1}) x_{ij}(a) , \quad h_{ij}(a) = w_{ij}(a) w_{ij}(-1).
\]
Then we have
\[
\phi(w_{12}(a))  = \begin{pmatrix} 0 & a \\ -a^{-1} & 0 \end{pmatrix}
\quad \text{and} \quad
\phi(h_{12}(a))  = \begin{pmatrix} a & 0 \\ 0 & a^{-1} \end{pmatrix}.
\]

\begin{definition}
For commuting units $a$ and $b$ in $R$, the Steinberg symbol $\{ a, b \} \in K_2(R)$ is
\[
\{ a, b \} = h_{12}(a) h_{13}(b) h_{12}(a)^{-1} h_{13}(b)^{-1}.
\]
\end{definition}

\begin{proposition}[see Lemma 4.2.14 and Theorem 4.2.17 of \cite{Rosenberg}] \label{Steinberg-relations} For any units $a$, $a_1$, $a_2$, and $b$ in $R$ such that $b$ commutes with $a$, $a_1$, and $a_2$, we have
\begin{enumerate}
\item $\{ a,b \} = \{ a,b \}^{-1}$
\item $\{ a_1 a_2, b \} = \{ a_1, b \} \{ a_2, b \}$
\item $\{ a, 1-a \} = 1$, whenever $1-a$ is invertible
\end{enumerate}
\end{proposition}

The boundary map $\partial : K_{2}(R/I) \xrightarrow{} K_1(R,I)$ is defined as follows.  Any element $u \in K_2(R/I)$ can be expressed in terms of generators $x_{ij}(u_k) \in St(R/I)$.  We obtain an element $r \in St(R)$ by lifting each $u_k \in R/I$ to an element $r_k \in R$.  Then $\phi(r) \in GL(R,I)$, and we define $\partial (u)$ to be the image of $\phi(r)$ in $K_1(R,I)$.  One then checks that this is independent of the choice of lifts.

Now we specialize to the case $R = \mathcal L$ and $I = \mathcal L^{1}$.  

\begin{definition}
For any commuting units $a$ and $b$ in $\mathcal L / \mathcal L^1$, 
the determinant invariant $d(a,b)$ is the nonzero number
\[
d(a,b) = \det \partial \{ a,b \}.
\]
In particular, $d(a,b)$ satisfies the relations in Proposition \ref{Steinberg-relations}.
\end{definition}

In calculating the determinant invariant, $a \in \mathcal L / \mathcal L^1$ is lifted to an operator $A \in \mathcal L$, which is necessarily Fredholm, and $a^{-1}$ is lifted to a parametrix $Q$ of $A$ modulo trace class.  Here, and in the sequel, a parametrix modulo an ideal is an inverse modulo that ideal.  Thus, $I-AQ$ and $I-QA$ are trace class operators.  When $a$ and $b$ have invertible lifts, we have the following observation of Brown \cite{Brown}.

\begin{proposition}
Let $a$ and $b$ be commuting units in $\mathcal L / \mathcal L^1$.  If $a$ and $b$ have invertible lifts $A, B \in \mathcal L$, then $d(a,b) = \det A B A^{-1} B^{-1}$.  In particular, the determinant of the multiplicative commutator depends only on the Steinberg symbol $\{ a,b \}$.
\end{proposition}

\begin{proof}
Since $A$ and $B$ are invertible, $a^{-1}$ and $b^{-1}$ can be lifted to $A^{-1}$ and $B^{-1}$.  Thus
\begin{align*}
d(a,b)
& = \det \partial \, h_{12}(\sigma(A)) h_{13}(\sigma(B)) h_{12}(\sigma(A))^{-1} h_{13}(\sigma(B))^{-1} \\
& = \det \begin{pmatrix} A & 0 & 0 \\ 0 & A^{-1} & 0 \\ 0 & 0 & I \end{pmatrix}
\begin{pmatrix} B & 0 & 0 \\ 0 & I & 0 \\ 0 & 0 & B^{-1} \end{pmatrix}
\begin{pmatrix} A & 0 & 0 \\ 0 & A^{-1} & 0 \\ 0 & 0 & I \end{pmatrix}^{-1}
\begin{pmatrix} B & 0 & 0 \\ 0 & I & 0 \\ 0 & 0 & B^{-1} \end{pmatrix}^{-1} \\
& = \det A B A^{-1} B^{-1}. \qedhere
\end{align*}
\end{proof}

\begin{example}
Suppose $A = \exp \alpha, B = \exp \beta$ for operators $\alpha$ and $\beta$ with $[\alpha, \beta] \in \mathcal L^1$. Then
\[
d(a,b) = \exp \tr \, [\alpha, \beta]
\]
by the Helton-Howe-Pincus formula.  For example, if $f$ and $g$ are smooth functions on the the unit circle, then the Toeplitz operators $T_{e^f}$ and $T_{e^g}$ are Fredholm and commute modulo trace class.  One can use the Berger-Shaw formula to show that
\[
d(T_{e^f} + \mathcal L^1, T_{e^g} + \mathcal L^1) = \exp \frac{1}{2\pi i} \int f \, dg.
\]
\end{example}

\begin{lemma} \label{lemma:di-triviality}
For any $a \in R^\times$, we have $\{ a, 1 \} = \{ 1, a \} = 1 \in K_2(R)$. Hence, for any invertible $a \in \mathcal L / \mathcal L^1$, we have $d(a,1) = d(1,a) = 1$.
\end{lemma}

\begin{proof}
We note that $h_{13}(1) = w_{13}(1) w_{13}(-1) = 1$, so
\[
\{ a,1 \} = h_{12}(a) h_{13}(1) h_{12}(a)^{-1} h_{13}(1)^{-1} = 1,
\]
and similarly for $\{1,a\}$.
\end{proof}

\begin{lemma} \label{lemma:di-direct-sum}
For any commuting units $a_i$ and $b_i$ in $\mathcal L / \mathcal L^1$, $i=1,2$, we have
\[
d(a_1 \oplus a_2, b_1 \oplus b_2) = d(a_1, b_1) \, d(a_2, b_2).
\]
\end{lemma}

\begin{proof}
In calculating the boundary map, $a_i$ is lifted to a Fredholm operator $A_i$, and $a_i^{-1}$ is lifted to any parametrix $Q_i$ of $A_i$ modulo trace class.
Then we can lift $a_1 \oplus a_2$ to $A_1 \oplus A_2$, and $(a_1 \oplus a_2)^{-1}$ to $Q_1 \oplus Q_2$.
Similarly for $b_i$, $b_1 \oplus b_2$, and $(b_1 \oplus b_2)^{-1}$.
The result then follows since determinants are multiplicative over direct sums.
\end{proof}

\subsection{Joint torsion of commuting operators} \label{subsec:preliminaries-jt}

Let $(V_\bullet, d_\bullet)$ be an exact sequence of finite dimensional vector spaces,
\[
0 \xrightarrow{} V_n \xrightarrow{d_n} V_{n-1} \xrightarrow{d_{n-1}} V_{n-2} \xrightarrow{} \cdots \xrightarrow{} V_0 \to 0.
\]
Denote by $ \det V_k$ the top exterior power $\Lambda^{\dim V_k} V_k$, and define the determinant line
\[
\det V_\bullet = \det V_{n}^* \otimes \det V_{n-1} \otimes \det V_{n-2}^* \otimes \dots
\]
For each $k$, pick a nonzero element
$t_k \in \Lambda^{\mathrm{rank} \, d_k}V_{k}$
such that $d_k t_k \neq 0$.
By exactness, $ d_{k} t_{k} \wedge t_{k-1} $ is a nonzero element of $\det V_{k-1}$, so we make the following definition.
\begin{definition}
The torsion $\tau(V_\bullet, d_\bullet)$ of the complex $(V_\bullet, d_\bullet)$ is the volume element
\[
\tau(V_\bullet, d_\bullet) = (t_n)^* \otimes (d_n t_n \wedge t_{n-1}) \otimes (d_{n-1} t_{n-1} \wedge t_{n-2})^* \otimes \dots \in \det V_\bullet
\]
This is nonzero and is independent of the choices of the $t_k$ \cite{Knudsen-Mumford}.  Hence $\tau(V_\bullet, d_\bullet)$ defines a canonical generator of $\det V_\bullet$.
\end{definition}

For a vector space $W$, we will make frequent use of the isomorphism $\Lambda^k W^* \otimes \Lambda^k W \cong \mathbf C$ given by
\begin{equation} \label{eq:duality}
(v_1^* \wedge v_2^* \wedge \dots \wedge v_k^*) \otimes (w_1 \wedge w_2 \wedge \dots \wedge w_k) \mapsto \det \left( v_i^*(w_j) \right).
\end{equation}

\begin{example} \label{example:torsion=determinant}
The torsion of an isomorphism of a finite dimensional vector space is its determinant.
\end{example}

For a collection of commuting operators $A=(A_1,\dots,A_n)$ on a vector space $\mathscr H$, Carey and Pincus \cite{Joint} and Kaad \cite{Kaad} have defined invariants known as joint torsion.  Let us review their constructions.  The Koszul complex $K_\bullet(A)$ is the chain complex with
\[
K_i(A) = \mathscr H \otimes \Lambda^i \mathbf{C}^n
\]
and differential $d_i:K_i(A) \to K_{i-1}(A)$ given by
\begin{equation*}
d_i = \sum_{k=1}^n A_k \otimes \varepsilon_k^*
\end{equation*}
where $\varepsilon_k:\Lambda^i \mathbf{C}^n \to \Lambda^{i+1} \mathbf{C}^n$ is the operation of exterior multiplication by the unit vector $e_k$, and $\varepsilon_k^*$ is its adjoint.
The $n$-tuple $A$ is said to be Fredholm if $(K_\bullet(A), d_\bullet)$ has finite dimensional homology.
We also have maps $\iota_k: \Lambda^i \mathbf{C}^n \to \Lambda^i \mathbf{C}^{n+1}$ induced by the inclusion
\[
\mathbf C^n \to \mathbf C^{n+1}, \quad (a_1,\dots,a_n) \mapsto (a_1,\dots,a_{k-1},0,a_k,\dots,a_n).
\]

For any $j \in \{ 1, \dots, n\}$, let $j(A) = (A_1, \dots, \hat A_j, \dots, A_n)$.  Since $A=(A_1,\dots,A_n)$ is commutative, each $A_j$ defines a morphism of chain complexes 
$A_j := A_j \otimes 1: K_\bullet(j(A)) \to K_\bullet(j(A))$.  
The mapping cone of this morphism is isomorphic to the Koszul complex $K_\bullet(A)$ via the isomorphism
\begin{equation*} \label{eq:iso-Koszul}
\begin{pmatrix} \iota_j^* \varepsilon_j^* \\ \iota_j^* \end{pmatrix} : \Lambda^k \mathbf{C}^n \xrightarrow{\sim} \Lambda^{k-1} \mathbf{C}^{n-1} \oplus \Lambda^k \mathbf{C}^{n-1}.
\end{equation*}
We thus obtain a triangle of chain complexes
\[
K_\bullet(j(A)) \to K_\bullet(j(A)) \to K_\bullet(A) \to
\]
and hence a long exact sequence in homology
\[
\mathcal E_j: 0 \to H_n(A) \to H_{n-1}(j(A)) \xrightarrow{} H_{n-1}(j(A)) \to H_{n-1}(A) \to \dots
\]

The torsion vector 
$\tau(\mathcal E_j) \in \det H_\bullet(A) \otimes \det H_\bullet(j(A))^* \otimes \det H_\bullet(j(A))$.
Here, and in the sequel, for a sequence of homology spaces $H_\bullet(A)$, we use the notation
\[
\det H_\bullet(A) = \det H_n(A)^* \otimes \det H_{n-1}(A) \otimes \dots
\]
We regard $\tau(\mathcal E_j) \in \det H_\bullet(A)$ using the isomorphism in \eqref{eq:duality} applied to the line
$\det H_\bullet(j(A))$.
Carrying out the same process for some $i \neq j$ yields an exact sequence $\mathcal E_i$.  If $i(A)$ and $j(A)$ are Fredholm, then $\mathcal E_i$ and $\mathcal E_j$ consist of finite dimensional vector spaces, whence we can form two torsion vectors 
$\tau(\mathcal E_i), \tau(\mathcal E_j) \in \det H_\bullet(A)$.  
Then $\tau(\mathcal E_i) \otimes \tau(\mathcal E_j)^*$ can be identified with a nonzero scalar, which up to a sign is the joint torsion defined by Carey and Pincus for $n=2$ \cite{Joint} and by Kaad for $n\geq 2$ \cite{Kaad}.

Let us describe this construction more explicitly in the case $n=2$.  Thus, let $A_1 = A$ and $A_2 = B$ be two commuting Fredholm operators.  Upon choosing bases for the exterior algebras, we have
\[
K_\bullet(A): \mathscr H \xrightarrow{A} \mathscr H, \quad
K_\bullet(B): \mathscr H \xrightarrow{B} \mathscr H
\]
\[
K_\bullet(A,B) : \mathscr H \xrightarrow{\left( \begin{smallmatrix} -B \\ A \end{smallmatrix} \right) } \mathscr H^2 \xrightarrow{ \left( \begin{smallmatrix} A & B \end{smallmatrix} \right) } \mathscr H
\]
Then $H_0(A) = \coker \, A$, $H_1(A) = \ker A$, $H_2(A,B) = \ker A \cap \ker B$, $H_0(A,B) = \mathscr H / (A \mathscr H + B \mathscr H)$, and
\[
H_1(A,B) = \frac{ \{ (y,z) \, | \, Ay+Bz = 0 \} }{ \{ (-Bx,Ax) \, | \, x \in \mathscr H \} }.
\]
The two exact sequences $\mathcal E_A$ and $\mathcal E_B$ are given by
\begin{multline*}
\mathcal E_{A}:
0 \xrightarrow{} (\ker A \, \cap \, \ker B) \xrightarrow{\iota} \ker B \xrightarrow{A} \ker B \xrightarrow{\iota_{2*}} H_1(A,B) \to \\
\xrightarrow{\pi_{1*}} \coker \, B \xrightarrow{A} \coker \, B \xrightarrow{\pi} \mathscr H / (A\mathscr H + B\mathscr H) \xrightarrow{} 0
\end{multline*}
\begin{multline*}
\mathcal E_{B}:
0 \xrightarrow{} (\ker A \, \cap \, \ker B) \xrightarrow{-\iota} \ker A \xrightarrow{B} \ker A \xrightarrow{\iota_{2*}} H_1(A,B) \to \\
\xrightarrow{\pi_{1*}} \coker \, A \xrightarrow{B} \coker \, A \xrightarrow{\pi} \mathscr H / (A\mathscr H + B\mathscr H) \xrightarrow{} 0
\end{multline*}
The map $\iota_{k*}$ is induced by inclusion into the $k$-th coordinate, and $\pi_{k*}$ is induced by projection onto the $k$-th coordinate.  Here, we have chosen signs to simplify formulas.

\begin{definition} \cite[Definition 3]{Joint}
The joint torsion $\tau(A,B)$ of $A$ and $B$ is the nonzero scalar
\[
\tau(A,B) = (-1)^{\lambda(A,B)} \, \tau(\mathcal E_A) \otimes \tau(\mathcal E_B)^*
\]
where $\lambda(A,B)$ depends on the homology spaces according to
\begin{align*}
\lambda(A,B) 
={}& \dim (\ker A \cap \ker B) \cdot \big( \dim \ker A + \dim \ker B \big) + \\
{}& + \dim (\mathscr H / (A \mathscr H + B \mathscr H)) \cdot ( \dim \coker \, A + \dim \coker \, B ).
\end{align*}
\end{definition}

\begin{lemma} \label{lemma:mult-lefschetz}
If $H_i(A,B) = 0$ for $i=0,1,2$, then we have
\[
\tau(A,B) = \frac{\det B|_{\ker A}}{\det B|_{\coker \, A}} \, \frac{\det A|_{\coker \, B}}{\det A|_{\ker B}}.
\]
In particular, $\tau(A,I) = \tau(I,A) = 1$.
\end{lemma}

\begin{proof}
Since $H_i(A,B)=0$, the exact sequence $\mathcal E_{A}$ breaks up into the isomorphisms
\[
A|_{\ker B} : \ker B \to \ker B \quad \text{and} \quad A|_{\coker \, B} : \coker \, B \to \coker \, B
\]
and $\mathcal E_{B}$ breaks up as
\[
B|_{\ker A} : \ker A \to \ker A \quad \text{and} \quad B|_{\coker \, A} : \coker \, A \to \coker \, A.
\]
Joint torsion is the alternating product of the torsion vectors of these isomorphisms, which are simply the determinants by Example \ref{example:torsion=determinant}.
\end{proof}

\begin{example}
Joint torsion generalizes the index of a Fredholm operator.  In fact,
\begin{align*}
\tau(A, \exp B)
& = \frac{\det \exp B|_{\ker A}}{\det \exp B|_{\coker \, A}}  \\
& = \exp \tr (B|_{\ker A} - B|_{\coker \, A} )
\end{align*}
by Lemma \ref{lemma:mult-lefschetz}.  Thus, $\tau(A, \exp B)$ is the exponential of the Lefschetz number of $B$ as an endomorphism of the chain complex $K_\bullet(A)$.  By taking $B=I$, we find
$\ind \, A = \log \tau(A,eI)$.
\end{example}

\subsection{Joint torsion of almost commuting operators} \label{subsec:preliminaries-almost}

We would like to calculate joint torsion for operators $A$ and $B$ that do not necessarily commute with one another.  The difficulty is that there is no longer a well-defined Koszul complex $K_\bullet(A,B)$.  Carey and Pincus circumvent this problem in the case $n=2$ by introducing auxiliary operators $C$ and $D$ that are perturbations of the original operators $A$ and $B$.  Thus, consider two Fredholm operators $A$ and $B$ with $[A,B] \in \mathcal L^1$.  
Further, assume the existence of operators $C$ and $D$ such that $AB=CD, A-D \in \mathcal L^1$, and $B-C \in \mathcal L^1$.  See Remark \ref{remark:auxiliary}.  Then we have the following commutative diagram:
\begin{equation*}
\begin{CD}
\mathscr H @>B>> \mathscr H \\
@VDVV @VVAV \\
\mathscr H @>C>> \mathscr H
\end{CD}
\end{equation*}

As in the commuting case, we can consider the mapping cone $K_\bullet(A,B,C,D)$ of the vertical chain map $(A,D)$.  Explicitly,
\[
K_\bullet(A,B,C,D) : \mathscr H \xrightarrow{\left( \begin{smallmatrix} -B \\ D \end{smallmatrix} \right) } \mathscr H^2 \xrightarrow{ \left( \begin{smallmatrix} A & C \end{smallmatrix} \right) } \mathscr H
\]
This yields a triangle of modified Koszul complexes, and hence the long exact sequence in homology:
\begin{multline*}
\mathcal E_{A,D}:
0 \xrightarrow{} (\ker B \, \cap \, \ker D) \xrightarrow{\iota} \ker B \xrightarrow{D} \ker C \xrightarrow{\iota_{2*}} H_1(A,B,C,D) \to \\
\xrightarrow{\pi_{1*}} \coker \, B \xrightarrow{A} \coker \, C \xrightarrow{\pi} \mathscr H / (A\mathscr H + C\mathscr H) \xrightarrow{} 0
\end{multline*}
The homology space $H_1(A,B,C,D)$ is given by
\[
H_1(A,B,C,D) = \frac{ \{ (y,z) \, | \, Ay+Cz = 0 \} }{ \{ (-Bx,Dx) \, | \, x \in \mathscr H \} }.
\]
The map $\iota_{2*}$ is induced by inclusion into the second coordinate, and $\pi_{1*}$ is induced by projection onto the first coordinate.  Similarly, we can consider the mapping cone of the horizontal chain map $(B,C)$, and hence the long exact sequence:
\begin{multline*}
\mathcal E_{B,C}:
0 \xrightarrow{} (\ker B \, \cap \, \ker D) \xrightarrow{-\iota} \ker D \xrightarrow{B} \ker A \xrightarrow{\iota_{1*}} H_1(A,B,C,D) \to \\
\xrightarrow{\pi_{2*}} \coker \, D \xrightarrow{C} \coker \, A \xrightarrow{\pi} \mathscr H / (A\mathscr H + C\mathscr H) \xrightarrow{} 0
\end{multline*}

However, $\tau(\mathcal E_{A,D}) \otimes \tau(\mathcal E_{B,C})^* \in \det H_\bullet(A) \otimes \det H_\bullet(D)^* \otimes \det H_\bullet(B) \otimes \det H_\bullet(C)^*$ is no longer canonically identified with a scalar.  To obtain a scalar, Carey and Pincus produce canonical generators of these determinant lines, known as perturbation vectors \cite[Section 3]{Perturbation}.  We will find it convenient to give a slightly different definition.  Proposition \ref{prop:perturbation-agreement-n=2} below shows that these two definitions agree.

Thus, let $A$ and $A'$ be Fredholm operators such that $A-A' \in \mathcal L^1$.  First assume that $A$ (and hence also $A'$) has index zero.
Let $L$ and $L'$ be trace class operators such that 
\begin{enumerate}
\item $A + L$ and $A' + L'$ are invertible
\item $L(\ker A) \cap \im \, A = L'(\ker A') \cap \im \, A' = \{ 0 \}$
\end{enumerate}
Then $L$ and $L'$ induce isomorphisms
\begin{align*}
\pi L = \pi \circ L|_{\ker A} & : \ker A \xrightarrow{\sim} \coker \, A \\
\pi' L' = \pi' \circ L'|_{\ker A'} & : \ker A' \xrightarrow{\sim} \coker \, A'
\end{align*}
where $\pi$ and $\pi'$ are the quotient maps by $\im \, A$ and $\im \, A'$, respectively.  Let $\tau(\pi L)$ and $\tau(\pi' L')$ be the torsion vectors of the above isomorphisms. Let $i : \left( \ker A \right)^\perp \hookrightarrow \mathscr H$ be the inclusion, let $P: \mathscr H \to \im \, A$ be the continuous projection along $L(\ker A)$, and define
\[
D(L) = \det \left[ I + i ( P A i)^{-1} ( P L ) \right].
\]

Now in general, if $A$ and $A'$ have nonzero index, let $Q$ be a Fredholm operator with index negative that of $A$.  Then $A\oplus Q$ and $A' \oplus Q$ are Fredholm operators of index zero on $\mathscr H \oplus \mathscr H$.  Choose $L$ and $L'$ as above, now for $A \oplus Q$ and $A' \oplus Q$, respectively.  Since $( A \oplus Q + L )^{-1}$ is a parametrix for $( A' \oplus Q + L' )$ modulo $\mathcal L^1(\mathscr H \oplus \mathscr H)$, we have the invertible determinant class operator
\[
\Sigma = ( A\oplus Q + L )^{-1} ( A' \oplus Q + L' ).
\]

\begin{definition} \label{def:perturbation-vectors}
The perturbation vector $\sigma_{A,A'}$ is the element in $\det H(A) \otimes \det H(A')^*$ defined by
\[
\sigma_{A,A'} =D(L) D(L')^{-1} \det \Sigma \cdot \tau(\pi L) \otimes \tau(\pi' L')^*.
\]
\end{definition}

\begin{proposition} \label{prop:perturbation-agreement-n=2}
The above definition agrees with that of Carey and Pincus \cite[Section 3, Equation 41]{Perturbation}.  In particular, $\sigma_{A,A'}$ is independent of the choices of $Q$, $L$, and $L'$.
\end{proposition}

\begin{proof}
First we note that this definition is exactly Carey and Pincus's perturbation vector $\sigma_{A\oplus Q, A' \oplus Q}$ for $A\oplus Q$ and $A' \oplus Q$, provided that we choose $L$ so that its range is linearly independent from that of $A \oplus Q$, and similarly for $L'$.

Now suppose that $L$ and $L'$ only satisfy conditions (1) and (2) above.  Then $(I - P)L$ and $(I-P')L'$ are as in the preceding paragraph, where $P$ is the projection onto $\im ( A \oplus Q)$.  Let $\iota$ be the inclusion $\left( \ker A \oplus Q \right)^\perp \hookrightarrow \mathscr H \oplus \mathscr H$.  We calculate
\begin{align*}
\left( A \oplus Q + (I-P)L \right)^{-1} \left( A \oplus Q + L \right) & = I + \left( A \oplus Q + (1-P)L \right)^{-1} PL\\
& =  I + i ( P ( A \oplus Q ) i)^{-1} ( P L ).
\end{align*}
The determinant is $D(L)$, and a similar calculation holds for $D(L')$.  Hence we find that this definition agrees with Carey and Pincus's definition of $\sigma_{A\oplus Q, A' \oplus Q}$.  In particular, it is independent of the choices of $L$ and $L'$ by \cite[Theorem 11]{Perturbation}.

Next we show that specific choices of $L$ and $L'$ (and hence all choices) recover Carey and Pincus's perturbation vector $\sigma_{A,A'}$.
First assume $\mathrm{ind} \, A \geq 0$.  Write $\mathrm{ker} \, A = X_0 \oplus X_1$ with $\dim \mathrm{coker} \, A = \dim X_0$, and write $\mathrm{im} \, Q^\perp = \tilde X_0 \oplus \tilde X_1$ with $\dim \mathrm{ker} \, Q = \dim \tilde X_0$.  Choose subspaces $X_0', X_1'$ for $A'$ similarly.  Pick isomorphisms 
$L_{11}: X_0 \xrightarrow{} \mathrm{im} \, A^\perp$, 
$L_{11}': X_0' \xrightarrow{} \mathrm{im} \, A'^\perp$, 
$L_{22}: \mathrm{ker} \, Q \xrightarrow{} \tilde X_0$, 
$L_{21}: \mathrm{ker} \, (A+L_{11}) \to \mathrm{im} \, (Q+L_{22})^\perp$, and 
$N: \mathrm{ker} \, (A'+L_{11}') \to \mathrm{ker} \, (A+L_{11})$.  Extend these operators by zero to all of $H$.  Define the operators
\begin{equation*}
L =
\begin{pmatrix}
L_{11} & 0 \\
L_{21} & L_{22}
\end{pmatrix}
\qquad
L' =
\begin{pmatrix}
L_{11}' & 0 \\
L_{21} N & L_{22}
\end{pmatrix}
\end{equation*}
Then $A \oplus Q+L$ and $A' \oplus Q+L'$ are invertible.  Moreover, 
\begin{align*}
\Sigma
& = ( A\oplus Q + L )^{-1} ( A' \oplus Q + L' ) \\
& = \begin{pmatrix} (A+L_{11})^R & L_{21}^{-1} \\ 0 & (Q+L_{22})^L \end{pmatrix} \begin{pmatrix} A' + L_{11}' & 0 \\  L_{21} N & Q + L_{22} \end{pmatrix}
\end{align*}
so $\det \Sigma = \det \left( (A+L_{11})^R (A' + L_{11}' )+N \right)$.
Here, $(A+L_{11})^R$ denotes the right inverse such that
\[
(A+L_{11})^R (A+L_{11}) = \mathrm{projection \, \, onto} \, \, X_1^\perp.
\]
Similarly for the left inverse.  This agrees with the corresponding term in Carey and Pincus's definition.  The factors associated with $Q$ in $\tau(L)$ and $\tau(L')$ cancel by the identity $v^*(v) = 1$.  We are then left with the definition of Carey and Pincus.  The case when $\ind \, A \leq 0$ is similar.  The last statement in the proposition follows from \cite[Theorem 11]{Perturbation}.
\end{proof}

\begin{definition} \label{def:almost-commuting-joint-torsion}
Suppose $A,B,C$, and $D$ are Fredholm operators with $AB=CD, A-D \in \mathcal L^1$, and $B-C \in \mathcal L^1$.  The joint torsion $\tau(A,B,C,D)$ is the nonzero scalar
\[
\tau(A,B,C,D) = (-1)^{\lambda(A,B,C,D)} \tau(\mathcal E_{A,D}) \otimes \tau(\mathcal E_{B,C})^* \otimes \sigma_{A,D} \otimes \sigma_{B,C}
\]
where $\lambda(A,B,C,D)$ depends on the homology spaces according to
\begin{align*}
\lambda(A,B,C,D) 
={}& \dim (\ker B \cap \ker D) \cdot \big( \dim \ker D + \dim \ker B \big) + \\
{}& + \dim (\mathscr H / (A \mathscr H + C \mathscr H)) \cdot ( \dim \coker \, A + \dim \coker \, C ).
\end{align*}
\end{definition}

In the commuting case, we have $\tau(A,B,B,A) = \tau(A,B)$ since $\sigma_{A,A} = 1$ and $\sigma_{B,B} = 1$.  By Proposition \ref{prop:perturbation-agreement-n=2} and the exact sequences $\mathcal E_{A,D}$ and $\mathcal E_{B,C}$, we have:

\begin{proposition}
The above definition agrees with that of Carey and Pincus \cite[Section 5, Equation 51]{Perturbation}.
\end{proposition}

\begin{lemma} \label{lemma:jt-direct-sum}
For any Fredholm operators $A_i, B_i, C_i, D_i$ with $A_iB_i = C_i D_i$, $A_i - D_i  \in \mathcal L^1$, and $B_i - C_i \in \mathcal L^1$, $i=1,2$, we have
\[
\tau (A_1 \oplus A_2, B_1 \oplus B_2, C_1 \oplus C_2, D_1 \oplus D_2) = \tau(A_1, B_1, C_1, D_1) \cdot \tau(A_2, B_2, C_2, D_2).
\]
\end{lemma}

\begin{proof}
In this case, we have
$\mathcal E_{A_1\oplus A_2, D_1 \oplus D_2} = \mathcal E_{A_1, D_1} \oplus \mathcal E_{A_2, D_2}$ and
$\mathcal E_{B_1\oplus B_2, C_1 \oplus C_2} = \mathcal E_{B_1, C_1} \oplus \mathcal E_{B_2, C_2}$.
The result follows by combining the torsion vectors of the sequences above with the perturbation vectors according to the definition, and by using the multiplicativity of perturbation vectors under direct sum.
\end{proof}

\begin{remark} \label{remark:auxiliary}
For $n\geq 2$ almost commuting operators, one can define joint torsion by introducing auxiliary operators as in Definition \ref{def:almost-commuting-joint-torsion}.  It is not known in what generality this can be done.
In fact, not every pair of almost commuting Fredholm operators $A$ and $B$ have trace class perturbations $D$ and $C$, respectively, such that $AB=CD$.  
For example, let $A$ be surjective but not injective, and let $B$ be a right inverse.
Suppose there exist operators $C$ and $D$ with $AB=CD$ and $A-D \in \mathcal L^1$, and $B-C \in \mathcal L^1$.
Then $CD=I$, so $C$ is surjective and $D$ is injective.  But this is impossible since $\ind \, C = \ind \, B < 0$ and $\ind \, D = \ind \, A > 0$.

On the other hand, auxiliary operators $C$ and $D$ can be found, for example, if either $\ind \, A \leq 0$ or $\ind \, B \geq 0$.  If $\ind \, A \leq 0$, we can pick a finite rank operator $F$ such that $A+F$ has a left inverse $(A+F)^{L}$, and
\[
AB = \big[ (A+F)B(A+F)^L - FB \big] \big[ A+F \big].
\]
The case when $\ind \, B \geq 0$ follows by taking adjoints.  Moreover, by Eschmeier's work \cite{Eschmeier}, there aways exist perturbations $A', A'', B', B''$ such that
\[
A'B = BA'' \quad \mathrm{and} \quad AB' = B''A.
\]
\end{remark}

\section{The finite dimensional case} \label{sec:finite-dimensional}

In this section, we show that joint torsion in a finite dimensional space is trivial.  Our proof relies on Kaad's comparison of vertical and horizontal torsion isomorphisms \cite[Theorem 4.3.3]{Kaad} and is essentially a special case of \cite[Theorem 5.1.1]{Kaad}.  Kaad has shown that whenever $A$ is a commuting $n$-tuple such that $ij(A) = (A_1, \dots, \hat A_i, \dots, \hat A_j, \dots, A_n)$ is Fredholm, the joint torsion $\tau_{i,j}(A)=1$.  We apply his results to the case of $n=2$ almost commuting operators.  In this case, the Fredholm condition just means that the space is finite dimensional.  First consider the exact sequence
\[
0 \xrightarrow{} \ker A \xrightarrow{} \mathscr H \xrightarrow{A} \mathscr H \xrightarrow{} \coker \, A \xrightarrow{} 0.
\]
If $\mathscr H$ is finite dimensional, the torsion vector $\tau(\mathcal A)$ of the above sequence is defined.

\begin{lemma} \label{lemma:perturbation-finite}
If $A$ and $D$ are linear operators on a finite dimensional space $\mathscr H$, then 
\[
\sigma_{A,D} =(-1)^{\kappa(A) + \kappa(D)} \tau(\mathcal A) \otimes \tau(\mathcal D)^*.
\]
where $\kappa(A) = \mathrm{nullity} \, A \cdot \mathrm{rank} \, A$ and $\kappa(D) = \mathrm{nullity} \, D \cdot \mathrm{rank} \, D$.
\end{lemma}

\begin{proof}
For $T = A,D$, choose a subspace $Z_T$ complementary to $\im \, T$, choose an isomorphism $F_T:\ker T \to Z_T$, and let $\pi_T: \mathscr H \to \coker \, T$ be the quotient map.  Then we have
\begin{align*}
\sigma_{A,D}
& = \det (A+F_A)^{-1} (D+F_D) \cdot \tau(\pi_A F_A) \otimes \tau(\pi_D F_D)^*.
\end{align*}
Here, the operators $F_A$ and $F_D$ appearing in the determinant have been extended by zero to the whole space $\mathscr H$.

On the other hand, choose any $t_0$ and $t_1$ with $t_0 \in \det \ker A$ and $t_1 \in \det \ker A^\perp$ nonzero.  Then we calculate
\[
\tau(\mathcal A) = t_0^* \otimes (t_0 \wedge t_1) \otimes (A t_1 \wedge F_A t_0)^* \otimes (\pi_A F_A t_0).
\]
The first and fourth factors form $\tau(\pi_A F_A)$.  For the middle factors,
\begin{align*}
(t_0 \wedge t_1) \otimes (A t_1 \wedge F_A t_0)^*
& = \left( (t_0 \wedge t_1)^* \otimes (A t_1 \wedge F_A t_0 ) \right)^{-1} \\
& = (-1)^{\dim \ker A \cdot \dim \im \, A} \left( (t_0 \wedge t_1)^* \otimes ( F_A t_0 \wedge A t_1) \right)^{-1} \\
& = (-1)^{\kappa(A)} \det (A+F_A)^{-1}.
\end{align*}
The same calculation for $\tau(\mathcal D)$ completes the proof.
\end{proof}

Now consider a double complex $\mathcal E_{\bullet \bullet}$ consisting of finite dimensional vector spaces 
$\mathcal E_{ij}$, $0 \leq i \leq m$, $0 \leq j \leq n$, 
with exact rows $\mathcal E_{i \bullet}$ and exact columns $\mathcal E_{\bullet j}$.  For each $i$, one can form the torsion vector $\tau(\mathcal E_{i \bullet})$ of row $i$ and combine all these to obtain the horizontal torsion vector
\[
\tau_h = \tau(\mathcal E_{0 \bullet})^* \otimes \tau(\mathcal E_{1 \bullet}) \otimes \dots
\]
One can similarly form the vertical torsion vector
\[
\tau_v =  \tau(\mathcal E_{\bullet 0})^* \otimes \tau(\mathcal E_{\bullet 1}) \otimes \dots
\]
Both of these vectors are generators of the determinant line of $\mathcal E_{\bullet \bullet}$, and moreover by \cite{Knudsen-Mumford}, we have:

\begin{proposition}[Knudsen-Mumford] \label{lemma:horizontal-vertical}
The horizontal and vertical torsion vectors agree, that is, $\tau_h = \tau_v$.
\end{proposition}

Next let $A$, $B$, $C$, and $D$ be operators on a finite dimensional vector space $\mathscr H$ such that $AB=CD$.  Consider the Koszul complexes $K_\bullet(A)$, $K_\bullet(B)$, $K_\bullet(C)$, $K_\bullet(D)$, and $K_\bullet(A,B,C,D)$.
At the level homology, we obtain the following commutative diagram of finite dimensional vector spaces with exact rows and columns.  Write $H_i = H_i(A,B,C,D)$ for the homology spaces.

\begin{equation*}
\begin{CD}
@. 0 @. 0 @. 0 @. \vdots\\
@. @VVV @VVV @VVV @VV\iota_{2*}V \\
0 @>>> H_2 @>>> \ker B @>D>> \ker C @>{\iota_{1*}}>> H_1 @>\pi_{2*}>> \cdots \\
@. @VVV @VVV @VVV @VV{\pi_{1*}}V \\
0 @>>> \ker D @>>> \mathscr H @>D>> \mathscr H @>>> \coker \, D @>>> 0 \\
@. @VVBV @VVBV @VVCV @VVCV\\
0 @>>> \ker A @>>> \mathscr H @>A>> \mathscr H @>>> \coker \, A @>>> 0 \\
@. @VV{\iota_{2*}}V @VVV @VVV @VVV \\
\cdots @>{\iota_{1*}}>> H_1 @>{\pi_{2*}}>> \coker \, B @>A>> \coker \, C @>>> H_0 @>>> 0 \\
@. @VV{\pi_{1*}}V @VVV @VVV @VVV \\
@. \vdots @. 0 @. 0 @. 0
\end{CD}
\end{equation*}

Here, the upper right and lower left corners are identified.  The above diagram therefore consists of three exact rows and three exact columns.  Two of the rows are the sequences from Lemma \ref{lemma:perturbation-finite} corresponding to $D$ and $A$, and the other row is the exact sequence $\mathcal E_{A,D}$ in the definition of joint torsion.  Two of the columns correspond to $B$ and $C$, and the other is the exact sequence $\mathcal E_{B,C}$.  Thus we obtain horizontal and vertical torsion vectors
\begin{align*}
\tau(\mathcal E_h) & = \tau(\mathcal E_{A,D})^* \otimes \tau(\mathcal D) \otimes \tau(\mathcal A)^* \\
\tau(\mathcal E^v) & = \tau(\mathcal E_{B,C})^* \otimes \tau(\mathcal B) \otimes \tau(\mathcal C)^*.
\end{align*}
As in Proposition \ref{lemma:horizontal-vertical}, we need to show that $\tau(\mathcal E_h)$ and $\tau(\mathcal E^v)$ agree, up to the signs in Definition \ref{def:almost-commuting-joint-torsion} and Lemma \ref{lemma:perturbation-finite}.
This will follow from \cite[Theorem 4.3.3]{Kaad}.

\begin{proposition} \label{prop:finite-dimensional}
If $\mathscr H$ is finite dimensional and $A,B,C$, and $D$ are operators on $\mathscr H$ such that $AB=CD$, then $\tau(A,B,C,D) = 1$.
\end{proposition}

\begin{proof}

Consider the following odd homotopy exact bitriangle of $\mathbf Z_2$-graded chain complexes.  See \cite[Equation 5.1]{Kaad}.

\begin{equation*}
\begin{CD}
\mathscr H @>B>> \mathscr H[1] @>\iota>> K_\bullet(B) @>{\pi}>> \mathscr H \\
@VDVV @V{-A}VV @V(D,A)VV @VDVV \\
\mathscr H[1] @>C>> \mathscr H @>\iota>> K_\bullet(C)[1] @>{-\pi}>> \mathscr H[1] \\
@V{\iota}VV @V{-\iota}VV @V{\iota}VV @V{\iota}VV \\
K_\bullet(D) @>(B,C)>> K_\bullet(A)[1] @>\iota>> K_\bullet(A,B,C,D) @>{\pi}>> K_\bullet(D) \\
@V{\pi}VV @V{-\pi}VV @V{\pi}VV @V{\pi}VV \\
\mathscr H @>B>> \mathscr H[1] @>\iota>> K_\bullet(B) @>{\pi}>> \mathscr H \\
\end{CD}
\end{equation*}
Here, $\mathscr H$ is given the grading with trivial odd part, and the notation $X[1]$ denotes the $\mathbf Z_2$-graded chain complex $X$ with the grading reversed and the differential negated.  Thus, the horizontal and vertical arrows are odd chain maps which anticommute with the differential, and the squares are anticommutative.

The rows $\mathcal F_B$, $\mathcal F_C$, and $\mathcal F_{B,C}$ are odd homotopy exact triangles, and so define torsion vectors $\tau(\mathcal F_B)$, $\tau(\mathcal F_C)$, and $\tau(\mathcal F_{B,C})$.  These correspond, respectively, to the exact sequences
\[
0 \to \ker B \xrightarrow{\iota} \mathscr H \xrightarrow{B} \mathscr H \xrightarrow{\pi} \coker \, B \to 0
\]
\[
0 \to \ker C \xrightarrow{\iota} \mathscr H \xrightarrow{C} \mathscr H \xrightarrow{-\pi} \coker \, C \to 0
\]
\begin{multline*}
0 \xrightarrow{} (\ker B \, \cap \, \ker D) \xrightarrow{-\iota} \ker D \xrightarrow{B} \ker A \xrightarrow{\iota_{1*}} H_1(A,B,C,D) \to \\
\xrightarrow{\pi_{2*}} \coker \, D \xrightarrow{C} \coker \, A \xrightarrow{\pi} \mathscr H / (A\mathscr H + C\mathscr H) \xrightarrow{} 0
\end{multline*}
Likewise, the columns $\mathcal F_D$, $\mathcal F_A$, and $\mathcal F_{A,D}$ are odd homotopy exact triangles, and so define torsion vectors $\tau(\mathcal F_D)$, $\tau(\mathcal F_A)$, and $\tau(\mathcal F_{A,D})$.   These correspond, respectively, to the exact sequences
\[
0 \to \ker D \xrightarrow{\iota} \mathscr H \xrightarrow{D} \mathscr H \xrightarrow{\pi} \coker \, D \to 0
\]
\[
0 \to \ker A \xrightarrow{-\iota} \mathscr H \xrightarrow{-A} \mathscr H \xrightarrow{-\pi} \coker \, A \to 0
\]
\begin{multline*}
0 \xrightarrow{} (\ker B \, \cap \, \ker D) \xrightarrow{\iota} \ker B \xrightarrow{D} \ker C \xrightarrow{\iota_{2*}} H_1(A,B,C,D) \to \\
\xrightarrow{\pi_{1*}} \coker \, B \xrightarrow{A} \coker \, C \xrightarrow{\pi} \mathscr H / (A\mathscr H + C\mathscr H) \xrightarrow{} 0
\end{multline*}
By \cite[Theorem 4.3.3]{Kaad}, the horizontal torsion
\[
\tau_h = \tau(\mathcal F_B)^* \otimes \tau(\mathcal F_C) \otimes \tau(\mathcal F_{B,C})
\]
agrees with the vertical torsion
\[
\tau_v = \tau(\mathcal F_D)^* \otimes \tau(\mathcal F_A) \otimes \tau(\mathcal F_{A,D})
\]
up to the sign of the permutation in \cite[Corollary 4.3.4]{Kaad}.  The torsion vectors 
$\tau(\mathcal F_B)$, $\tau(\mathcal F_C)$, $\tau(\mathcal F_{B,C})$, $\tau(\mathcal F_D)$, $\tau(\mathcal F_A)$, $\tau(\mathcal F_{A,D})$
evidently differ from the torsion vectors 
$\tau(\mathcal B)$, $\tau(\mathcal C)$, $\tau(\mathcal E_{B,C})$, $\tau(\mathcal D)$, $\tau(\mathcal A)$, $\tau(\mathcal E_{A,D})$, respectively,
only by the signs in their definitions.
Then, up to these signs, Lemma \ref{lemma:perturbation-finite} implies that $\sigma_{A,D} = \tau(\mathcal F_A) \otimes \tau(\mathcal F_D)^*$ and $\sigma_{B,C} = \tau(\mathcal F_B) \otimes \tau(\mathcal F_C)^*$.  Taking into account the sign in Definition \ref{def:almost-commuting-joint-torsion}, we find that
\begin{align*}
\tau(A,B,C,D)
& = \tau_v \otimes \tau_h^* \\
& = 1. \qedhere
\end{align*}
\end{proof}

\begin{corollary} \label{corollary:identity-plus-finite-rank}
Suppose $A,B,C$, and $D$ are operators on a Hilbert space $\mathscr H$, each of which differs from the identity by a finite rank operator.  If $AB=CD$, then $\tau(A,B,C,D) = 1.$
\end{corollary}

\begin{proof}
With respect to a decomposition
$\mathscr H = \mathscr H_0 \oplus V$
for some finite dimensional subspace $V$, the operators $A,B,C$ and $D$ are of the form
$I \oplus F_A, I \oplus F_B, I \oplus F_C$, and $I \oplus F_D$, respectively.  
Then $F_A F_B = F_C F_D$.  
By Lemma \ref{lemma:jt-direct-sum}, Lemma \ref{lemma:mult-lefschetz}, and Proposition \ref{prop:finite-dimensional}, respectively, we have
\begin{align*}
\tau(A, B, C, D)
& = \tau( I, I, I, I ) \cdot \tau( F_A, F_B, F_C, F_D ) \\
& = \tau(F_A, F_B, F_C, F_D) \\
& = 1. \qedhere
\end{align*}
\end{proof}

\section{Factorizations of perturbation vectors and joint torsion} \label{sec:factorization}

\subsection{Perturbation vectors}

For an invertible operator $U$ and finite dimensional subspace $V$ of $\mathscr H$, denote by $U|_{V}$ the isomorphism $U|_{V} : V \to U(V)$, and let $\tau(U|_V)$ denote the torsion of this isomorphism.  Also, for a Fredholm operator $T$, let $U|_{\coker \, T}$ denote the isomorphism $U|_{\coker \, T} : \coker \, T \to \coker \, UT$ given by $v + T\mathscr H \mapsto Uv + UT \mathscr H$.  Let $\tau(U|_{\coker \, T} )$ denote the torsion of this isomorphism.

\begin{lemma} \label{lemma:perturbation}
Let $a$ and $u$ be commuting units in $\mathcal L / \mathcal L^1$.  Let $A, D \in \mathcal L$ be lifts of $a$, and suppose $u$ has an invertible lift $U \in \mathcal L$.  Then we have
\[
\sigma_{A, U^{-1}DU} = d(a,u) \cdot \sigma_{A,D} \otimes \tau \left( U^{-1}|_{\ker D} \right) \otimes \tau \left(U^{-1}|_{\coker \, D} \right)^*.
\]
\end{lemma}

\begin{proof}

First we note that $A-U^{-1}DU \in \mathcal L^1$, so the perturbation vector $\sigma_{A, U^{-1}DU}$ is defined.  
We begin by proving the lemma in the case when $A$, and hence also $D$, has index zero.
For $T=A,D$, choose a subspace $Z_T$ complementary to $\im \, T$, and choose isomorphisms $F_T: \ker T \to Z_T$.  Let $\pi_T: \mathscr H \to \coker \, T$ be the quotient map.
The operator $U^{-1} F_D U$ defines an isomorphism
\begin{equation} \label{eq:invertible}
U^{-1} F_D U: \ker U^{-1}DU = U^{-1}(\ker D) \to U^{-1}(Z_D)
\end{equation}
and $U^{-1}(Z_D)$ is a subspace complementary to $\im ( U^{-1}DU ) = U^{-1} (\im \, D )$.
We calculate
\begin{align*}
\det (A+F_A)^{-1} (U^{-1}DU + U^{-1} F_D U)
={}& \det (A+F_A)^{-1} (D + F_D) \times \\
{}& \times \det (D + F_D)^{-1} U^{-1} (D + F_D) U \\
={}& \det (A+F_A)^{-1} (D + F_D) \cdot d(a,u).
\end{align*}
The torsion of the isomorphism induced by \eqref{eq:invertible} is given by
\[
\tau(\pi_{U^{-1}DU} U^{-1} F_D U) = \tau(\pi_D F_D) \otimes \tau \left( U^{-1}|_{\ker D} \right)^* \otimes \tau \left(U^{-1}|_{\coker \, D} \right).
\]
Combining the two preceding equations, we calculate
\begin{align*}
\sigma_{A, U^{-1}DU}
={}& \det (A+F_A)^{-1} (U^{-1}DU + U^{-1} F_D U) \cdot \tau(\pi_A F_A) \otimes \tau(\pi U^{-1} F_D U)^* \\
={}& d(a,u) \cdot \det (A+F_A)^{-1} (D + F_D) \cdot \tau(\pi_A F_A) \otimes \tau(\pi_D F_D)^* \otimes \\
{}& \otimes \tau \left( U^{-1}|_{\ker D} \right) \otimes \tau \left(U^{-1}|_{\coker \, D} \right)^* \\
={}&d(a,u) \cdot \sigma_{A,D}
\otimes \tau \left( U^{-1}|_{\ker D} \right) \otimes \tau \left(U^{-1}|_{\coker \, D} \right)^*.
\end{align*}

Now if $\ind \, A$ is not necessarily zero, let $Q$ be any Fredholm operator with $\ind \, Q = - \ind \, A$.  
Let $q$ be the image of in $\mathcal L / \mathcal L^1$ of $Q$.
Let $\tilde A = A \oplus Q$, $\tilde D = D \oplus Q$, and $\tilde U = U \oplus I$.
Then $\tilde A$ and $\tilde D$ have index zero, $\tilde A - \tilde D \in \mathcal L^1(\mathscr H^2)$, and $[\tilde D, \tilde U] \in \mathcal L^1(\mathscr H^2)$.  Hence we calculate
\begin{align*}
\sigma_{\tilde A, \tilde U^{-1} \tilde D \tilde U}
={}& d(\tilde a, \tilde u) \cdot \sigma_{\tilde A, \tilde D} \otimes \tau \left( \tilde U^{-1}|_{\ker \tilde D} \right) \otimes \tau \left( \tilde U^{-1}|_{\coker \, \tilde D} \right)^* \\
={}& d(a,u) \cdot d(q,1) \cdot \sigma_{A,D} \otimes \sigma_{S,S} \otimes \tau \left( U^{-1}|_{\ker D} \right) \otimes \tau \left(U^{-1}|_{\coker \, D} \right)^* \\
={}& d(a,u) \cdot \sigma_{A,D} \otimes \tau \left( U^{-1}|_{\ker D} \right) \otimes \tau \left(U^{-1}|_{\coker \, D} \right)^*.
\end{align*}
On the other hand,
\[
\sigma_{\tilde A, \tilde U^{-1} \tilde D \tilde U} = \sigma_{A, U^{-1}DU}
\]
so we have
\[
\sigma_{A, U^{-1}DU} = d(a,u) \cdot \sigma_{A,D} \otimes \tau \left( U^{-1}|_{\ker D} \right) \otimes \tau \left(U^{-1}|_{\coker \, D} \right)^*. \qedhere
\]
\end{proof}

\begin{lemma} \label{lemma:perturbation2}
Let $a$ and $u$ be units in $\mathcal L / \mathcal L^1$.  Let $A,D \in \mathcal L$ be lifts of $a$, and suppose $u$ has an invertible lift $U\in \mathcal L$.  Then we have
\begin{enumerate}
\item $\sigma_{AU, DU} = \sigma_{A,D} \otimes \tau \left( U^{-1}|_{\ker A} \right)^* \otimes \tau \left(U^{-1}|_{\ker D} \right)$
\item $\sigma_{UA, UD} = \sigma_{A,D} \otimes \tau \left( U|_{\coker \, A} \right) \otimes \tau \left(U|_{\coker \, D} \right)^*$
\end{enumerate}
\end{lemma}

\begin{proof}
First we note that $AU - DU, UA-UD \in \mathcal L^1$.
As before, let us begin with the case when $A$ has index zero.  Let $F_A, F_D, Z_A, Z_D$ be as in the proof of Lemma \ref{lemma:perturbation}.  Then
\[
\det (A+F_A U)^{-1} (D+F_D U) = \det (A+F_A)^{-1} (D+F_D).
\]
Also, $\ker AU = U^{-1} ( \ker A )$ and $\ker DU = U^{-1} ( \ker D )$, and we calculate
\begin{align*}
\tau(\pi_A F_A U)
={}& \tau(\pi_A F_A) \otimes \tau(U^{-1}|_{\ker A})^* \\
\tau(\pi_D F_D U)
={}& \tau(\pi_D F_D) \otimes \tau(U^{-1}|_{\ker D})^*
\end{align*}
Therefore $\sigma_{AU, DU} = \sigma_{A,D} \otimes \tau \left( U^{-1}|_{\ker A} \right)^* \otimes \tau \left(U^{-1}|_{\ker D} \right).$

In general, if $\ind \, A$ is not necessarily zero, let $Q$, $\tilde A$, $\tilde D$, $\tilde U$ be as in Lemma \ref{lemma:perturbation}.  Then
\[
\sigma_{\tilde A \tilde U, \tilde D \tilde U} = \sigma_{A,D} \otimes \tau \left( U^{-1}|_{\ker A} \right)^* \otimes \tau \left(U^{-1}|_{\ker D} \right).
\]
On the other hand,
\[
\sigma_{\tilde A \tilde U, \tilde D \tilde U} = \sigma_{AU, DU}.
\]
Therefore $\sigma_{AU, DU} = \sigma_{A,D} \otimes \tau \left( U^{-1}|_{\ker A} \right)^* \otimes \tau \left(U^{-1}|_{\ker D} \right).$
The second part is proved similarly.
\end{proof}

\begin{lemma} \label{lemma:perturbation-det-class}
Let $a$ and $u$ be units in $\mathcal L / \mathcal L^1$.  Let $A,D \in \mathcal L$ be lifts of $a$ and let $U \in \mathcal L$ be an invertible lift of $1 \in \mathcal L / \mathcal L^1$ (i.e.~$U$ is an invertible determinant class operator).  Then we have
\begin{enumerate}
\item $\sigma_{A,DU} = \sigma_{A,D} \otimes \tau \left( U^{-1}|_{\ker D} \right)^* \cdot \det U$
\item $\sigma_{A,UD} = \sigma_{A,D} \otimes \tau \left( U|_{\coker D} \right) \cdot \det U$
\end{enumerate}
\end{lemma}

\begin{proof}
First we note that $A-DU, A-UD \in \mathcal L^1$.
The proof of the lemma then proceeds as in the previous lemma.
\end{proof}

\subsection{Joint torsion}  In this section, we use the previous three lemmas to calculate the joint torsion of quadruples $(A,B,C,D)$ in terms of quadruples modified by an invertible operator.  This will allow us to reduce the calculation of joint torsion to a determinant invariant and a finite dimensional calculation, which has already been dealt with in Section \ref{sec:finite-dimensional}.

\begin{proposition} \label{prop:invertible}
Let $a$ and $b$ be commuting units in $\mathcal L / \mathcal L^1$.
Let $A,D \in \mathcal L$ be lifts of $a$, and let $B,C \in \mathcal L$ be lifts of $b$ such that $AB=CD$.
Suppose $u \in \mathcal L / \mathcal L^1$ has an invertible lift $U\in \mathcal L$.
\begin{enumerate}
\item If $a$ and $u$ commute, then we have
\[
\tau(A,BU, CU, U^{-1}DU) = d(a,u) \cdot \tau(A,B,C,D).
\]
\item If $b$ and $u$ commute, then we have
\[
\tau(UA,B, UCU^{-1}, UD) = d(u,b) \cdot \tau(A,B,C,D).
\]
\end{enumerate}
\end{proposition}

\begin{proof}
First we note that $A (BU) = (CU) (U^{-1}DU)$ and $A - U^{-1}DU \in \mathcal L^1$, and $BU-CU \in \mathcal L^1$.  Hence the joint torsion in (1) is defined, and similarly for (2).  Let us first prove (1).  By Lemmas \ref{lemma:perturbation} and \ref{lemma:perturbation2}, we have the factorization
\begin{multline} \label{eq:perturbation}
\sigma_{A,U^{-1}DU} \otimes \sigma_{BU,CU} = d(a,u) \cdot \sigma_{A,D} \otimes \sigma_{B,C}
\otimes \tau \left( U^{-1}|_{\ker B} \right)^*
\otimes \tau \left(U^{-1}|_{\ker C} \right) \otimes \\
\otimes \tau \left( U^{-1}|_{\ker D} \right)
\otimes \tau \left(U^{-1}|_{\coker \, D} \right)^*.
\end{multline}

To calculate $\tau(\mathcal E_{A, D})$, we choose generators $t_0, \dots, t_5$ appropriately:

\begin{itemize}
\item $t_0 \in \det (\ker B \cap \ker D)$
\item $t_0 \wedge t_1 \in \det \ker B$
\item $D t_1 \wedge t_2 \in \det \ker C$
\item $\iota_{2*} t_2 \wedge t_3 \in \det H_1(A,B,C,D)$
\item $\pi_{1*} t_3 \wedge t_4 \in \det \coker \, B$
\item $A t_4 \wedge t_5 \in \det \coker \, C$
\item $\pi t_5 \in \det \, (\mathscr H / (A \mathscr H + C \mathscr H ) )$
\end{itemize}
Then
$\tau(\mathcal E_{A, D}) = t_0^* \otimes (t_0 \wedge t_1) \otimes \dots \otimes (A t_4 \wedge t_5) \otimes (\pi t_5)^*.$
Here, $H_1 = H_1(A,B,C,D)$ is the first Koszul homology space
\[
H_1 = \frac{ \{ (y,z) \, | \, Ay = -Cz \} }{ \{ (-Bx, Dx) \, | \, x \in \mathscr H \} }.
\]
The map $\iota_{2*}$ is induced by inclusion into the second coordinate, and $\pi_{1*}$ is induced by projection onto the first coordinate.  Let $v \in \ker C$ be such that $v \notin D(\ker C)$.  Then $\iota_{2*} v = [(0,v)] \neq 0$ in $H_1$, so we can take $t_2$ to be the product of sufficiently many such vectors, say $\wedge_i v_i$.  On the other hand, let $w \notin \im \, B$ be such that $Aw = Cu \in \im \, C$ for some $u$.  Then $[(w, -u)] \neq 0$ in $H_1$ and $\pi_{1*} [(w,-u)]$, so we can take $t_3$ to be the product of sufficiently many such vectors, say $\wedge_j [(w_j, -u_j)]$.  Of course, $\tau(\mathcal E_{A, D})$ is independent of these specific choices.

Now we would like to calculate the torsion vectors $\tau(\mathcal E_{A, U^{-1}DU})$ and $\tau(\mathcal E_{BU, CU})$ in terms of $\tau(\mathcal E_{A, D})$ and $\tau(\mathcal E_{B, C})$.  The only potential difficulty is in the $H_1$ position, so let us compare $H_1' = H_1(A,BU,CU,U^{-1}DU)$ with the discussion of the previous paragraph.  First,
\[
H_1' = \frac{ \{ (y,U^{-1}z) \, | \, Ay = -Cz \} }{ \{ (-Bx, U^{-1}Dx) \, | \, x \in \mathscr H \} }
\]
The invertible operator $I \oplus U^{-1}$ on $\mathscr H \oplus \mathscr H$ induces an isomorphism from $H_1$ onto $H_1'$, which we denote by $( I \oplus U^{-1} )|_{H_1}$.   Moreover, we have
\begin{align*}
U^{-1} (D t_1 \wedge t_2)
& = U^{-1} D t_1 \wedge U^{-1} t_2 \\
& \in \det \ker CU \\
( I \oplus U^{-1} )|_{H_1} (\iota_{2*} t_2 \wedge t_3)
& = \left( \wedge_i [(0,U^{-1} v_i)] \right) \wedge \left( \wedge_j [(w_j, -U^{-1}u_j)] \right) \\
& = \iota_{2*} U^{-1} t_2 \wedge ( I \oplus U^{-1} )|_{H_1} t_3 \\
& \in \det H_1' \\
\pi_{1*} ( I \oplus U^{-1} )|_{H_1} t_3 \wedge t_4
& = \pi_{1*} t_3 \wedge t_4 \\
& \in \det \coker \, BU
\end{align*}
Hence, we calculate
\begin{align}
\tau(\mathcal E_{A, U^{-1}DU}) \notag
={}& (U^{-1}t_0)^* \otimes (U^{-1}t_0 \wedge U^{-1}t_1) \otimes (U^{-1}Dt_1 \wedge U^{-1}t_2)^* \otimes \\ \notag
{}& \otimes ((I \oplus U^{-1} )|_{H_1} \iota_{2*} t_2 \wedge t_3) \otimes (-\pi_{1*} t_3 \wedge t_4)^* \otimes (-A t_4 \wedge t_5) \otimes (\pi t_5)^*\\ \label{eq:torsion-AD}
={}& \tau(\mathcal E_{A, D}) \otimes
\tau((I\oplus U^{-1})|_{H_1} ) \otimes \tau \left( U^{-1}|_{\ker B} \right)
\otimes \tau \left(U^{-1}|_{\ker C} \right)^* \otimes \\ \notag
{}& \otimes \tau \left(U^{-1}|_{\ker B \cap \ker D} \right)^*.
\end{align}
Similarly, we find
\begin{align} \label{eq:torsion-BC}
\tau(\mathcal E_{BU, CU}) 
={}& \tau(\mathcal E_{B, C}) \otimes \tau((I\oplus U^{-1})|_{H_1} ) \otimes \tau \left( U^{-1}|_{\ker D} \right) \otimes \tau \left(U^{-1}|_{\coker \, D} \right)^* \otimes \\ \notag
{}& \otimes \tau \left(U^{-1}|_{\ker B \cap \ker D} \right)^*.
\end{align}

Notice that $\lambda = \lambda(A,B,C,D)$ is the same as $\lambda(A,BU,CU,U^{-1}DU)$ in Definition \ref{def:almost-commuting-joint-torsion} since all homology spaces have the same dimension.  Combining \eqref{eq:perturbation}, \eqref{eq:torsion-AD}, and \eqref{eq:torsion-BC} yields
\begin{align*}
\tau(A, BU, CU, U^{-1}DU)
& = (-1)^{\lambda} \tau(\mathcal E_{A, U^{-1}DU}) \otimes \tau(\mathcal E_{BU, CU})^* \otimes \sigma_{A,U^{-1}DU} \otimes \sigma_{BU,CU} \\
& = (-1)^{\lambda} d(a,u) \cdot \tau(\mathcal E_{A, D}) \otimes \tau(\mathcal E_{B, C})^* \otimes \sigma_{A,D} \otimes \sigma_{B,C} \\
& = d(a,u) \cdot \tau(A,B,C,D).
\end{align*}
This completes the proof of part 1.  The second part follows by a similar calculation.
\end{proof}

\begin{proposition} \label{prop:invertible2}
Let $A,B,C$, and $D$ be Fredholm operators with $AB=CD$, $A-D \in \mathcal L^1$, and $B-C \in \mathcal L^1$.
For any invertible determinant class operator $U$, we have
\[
\tau(A,B, CU, U^{-1}D) = \tau(A,B,C,D).
\]
\end{proposition}

\begin{proof}
First we note that $I-U^{-1} = (U-I) U^{-1} \in \mathcal L^1$, so $A-U^{-1}D \in \mathcal L^1$ and $B-CU \in \mathcal L^1$.
We calculate perturbation vectors as in \eqref{eq:perturbation}, this time using using Lemma \ref{lemma:perturbation-det-class}.  Indeed,
\begin{equation} \label{eq:perturbation2}
\sigma_{A,U^{-1}D} \otimes \sigma_{B,CU} = \sigma_{A,D} \otimes \sigma_{B,C} \otimes \tau(U^{-1}|_{\coker \, D})^* \otimes \tau(U|_{\ker \, C})^*. \\
\end{equation}

Next pick generators $t_0, \dots, t_5$ as in Proposition \ref{prop:invertible}.  As before, we find that
\begin{equation*}
H_1(A,B, CU, U^{-1}D) = \frac{ \{ (y,U^{-1}z) \, | \, Ay = -Cz \} }{ \{ (-Bx, U^{-1}Dx) \, | \, x \in \mathscr H \} }
\end{equation*}
The invertible operator $I \oplus U^{-1}$ on $\mathscr H \oplus \mathscr H$ induces an isomorphism from $H_1(A,B,C,D)$ onto $H_1(A,B, CU, U^{-1}D)$, which we denote by $( I \oplus U^{-1} )|_{H_1}$.
Moreover, we have
\begin{align*}
U^{-1} (D t_1 \wedge t_2)
& = U^{-1} D t_1 \wedge U^{-1} t_2 \\
& \in \det \ker CU \\
( I \oplus U^{-1} )|_{H_1} (\iota_{2*} t_2 \wedge t_3)
& = \iota_{2*} U^{-1} t_2 \wedge ( I \oplus U^{-1} )|_{H_1} t_3 \\
& \in \det H_1(A,B, CU, U^{-1}D) \\
\pi_{1*} ( I \oplus U^{-1} )|_{H_1} t_3 \wedge t_4
& = \pi_{1*} t_3 \wedge t_4 \\
& \in \det \coker \, B.
\end{align*}
Hence, we calculate
\begin{align} \label{eq:torsion-AD2}
\tau(\mathcal E_{A, U^{-1}D})
={}& \tau(\mathcal E_{A, D}) \otimes
\tau((I\oplus U^{-1})|_{H_1} )
\otimes \tau \left(U^{-1}|_{\ker C} \right)^*
\otimes \tau \left(U^{-1}|_{\ker B \cap \ker D} \right)^* \\
\tau(\mathcal E_{B, CU}) \label{eq:torsion-BC2}
={}& \tau(\mathcal E_{B, C}) \otimes \tau((I\oplus U^{-1})|_{H_1} ) \otimes \tau \left(U^{-1}|_{\coker \, D} \right)^* \otimes
\tau \left(U^{-1}|_{\ker B \cap \ker D} \right)^*.
\end{align}
Combining equations \eqref{eq:perturbation2}, \eqref{eq:torsion-AD2}, and \eqref{eq:torsion-BC2}, we find that
$\tau(A,B, CU, U^{-1}D) = \tau(A,B,C,D)$,
as desired.
\end{proof}

\subsection{A proof of Theorem \ref{thm:jt=di}}

We begin by using the propositions of the previous section to establish Theorem \ref{thm:jt=di} in the case of index zero operators.  The general case will then follow quickly.

\begin{proposition} \label{prop:jt=di-index-zero}
Let $a$ and $b$ be commuting units in $\mathcal L / \mathcal L^1$.  Suppose $a$ has index zero lifts $A,D \in \mathcal L$, and suppose $b$ has index zero lifts $B,C \in \mathcal L$, such that $AB=CD$.
Then $d(a,b) = \tau(A,B,C,D).$
\end{proposition}

\begin{proof}
Let $U$ and $V$ be invertible parametrices for $A$ and $B$, respectively, modulo finite rank operators.  
For example, we can take $U = (A+F)^{-1}$ for any suitable finite rank operator $F$.  
The images of $U$ and $V$ in $\mathcal L / \mathcal L^1$ are $a^{-1}$ and $b^{-1}$, respectively.
By Proposition \ref{prop:invertible}, we calculate
\begin{align*}
\tau(A,B,C,D)
& = d(a^{-1},b)^{-1} \cdot \tau(UA, B, UCU^{-1}, UD) \\
& = d(a^{-1},b)^{-1} \cdot d(1,b^{-1})^{-1} \cdot \tau(UA, BV, UCU^{-1}V, V^{-1}UDV).
\end{align*}
We have $d(a^{-1},b)^{-1} = d(a,b)$ and $d(1,b^{-1})^{-1} = 1$ by Proposition \ref{Steinberg-relations} and Lemma \ref{lemma:di-triviality}.  Thus
\[
\tau(A,B,C,D) = d(a,b) \cdot \tau(A', B', C', D')
\]
where
$A'=UA$ and $B'=BV$ differ from the identity by finite rank operators, and $C'=UCU^{-1}V$ and $D'=V^{-1}UDV$ differ from the identity by trace class operators.
Let $W$ be an invertible parametrix for $D'$ modulo finite rank, and hence an invertible determinant class operator.  Then by Proposition \ref{prop:invertible2}, we have
\[
\tau(A', B', C', D') = \tau(A', B', C'W^{-1}, WD').
\]
Since $A', B'$, and $WD'$ differ from the identity by finite rank operators, so does $C'W^{-1}$.
Therefore $\tau(A', B', C'W^{-1}, WD') = 1$ by Corollary \ref{corollary:identity-plus-finite-rank}, and the result follows.  
\end{proof}

We are now in a position to prove the equality of joint torsion and the determinant invariant:
\begin{theorem1}
Let $a$ and $b$ be commuting units in $\mathcal L / \mathcal L^1$.
Let $A, D \in \mathcal L$ be lifts of $a$, and let $B, C \in \mathcal L$ be lifts of $b$ such that $AB=CD$.
Then $d(a,b) = \tau(A,B,C,D)$.
\end{theorem1}

\begin{proof}
Pick any Fredholm operators $Q$ and $R$ with $\ind \, Q = -\ind \, A$ and $\ind \, R = -\ind \, B$.  Let $q$ and $r$ be the images in $\mathcal L / \mathcal L^1$ of $Q$ and $R$, respectively.  Consider the index zero operators $\tilde A = A \oplus Q \oplus I$, $\tilde B = B \oplus I \oplus R$, $\tilde C = C \oplus I \oplus R$, and $\tilde D = D \oplus Q \oplus I$.
By Lemmas \ref{lemma:jt-direct-sum} and \ref{lemma:mult-lefschetz}, we have
\begin{align*}
\tau( \tilde A, \tilde B, \tilde C, \tilde D)
& = \tau(A,B,C,D) \cdot \tau(Q,I,I,Q) \cdot \tau(I,R,R,I) \\
& = \tau(A,B,C,D).
\end{align*}
The operators $\tilde A$, $\tilde B$, $\tilde C$, and $\tilde D$ on the left hand side all have index zero, so by Proposition \ref{prop:jt=di-index-zero}, we have $\tau( \tilde A, \tilde B, \tilde C, \tilde D) = d( a \oplus q \oplus 1, b \oplus 1 \oplus r)$.  By Lemmas \ref{lemma:di-direct-sum} and \ref{lemma:di-triviality}, we have
\begin{align*}
d( a \oplus q \oplus 1, b \oplus 1 \oplus r)
& = d(a,b) \cdot d(q,1) \cdot d(1,r) \\
& = d(a,b).
\end{align*}
Therefore $\tau(A,B,C,D) = d(a,b)$.
\end{proof}

\begin{notation}
By Theorem \ref{thm:jt=di}, the joint torsion $\tau(A,B,C,D)$ depends only on the images of $A$ and $B$ in $\mathcal L / \mathcal L^1$, so we are justified in writing
\[
\tau(A,B,C,D) = \tau(a,b)
\]
where $a$ and $b$ are the images in $\mathcal L / \mathcal L^1$ of $A$ and $B$, respectively.
\end{notation}

\subsection{Consequences of Theorem \ref{thm:jt=di}}

Joint torsion enjoys the same properties as the determinant invariant by Theorem \ref{thm:jt=di}.  In particular, we find that joint torsion satisfies the Steinberg relations of Proposition \ref{Steinberg-relations}.

\begin{corollary}
Whenever the joint torsion numbers below are defined, we have the following:
\begin{enumerate}
\item $\tau(a_1 a_2, b) = \tau(a_1, b) \, \tau(a_2, b)$
\item $\tau(a, b) = \tau(b, a)^{-1}$
\item $\tau(a, 1-a) = 1$.
\end{enumerate}
\end{corollary}

Let $A_\lambda$ and $B_\lambda$ be continuous families of almost commuting Fredholm operators.  Assume that for each $\lambda$, the joint torsion $\tau(A_\lambda, B_\lambda)$ is defined, i.e.~there exist auxiliary operators $C_\lambda$ and $D_\lambda$ such that $A_\lambda - D_\lambda \in \mathcal L^1$, $B_\lambda - C_\lambda \in \mathcal L^1$, and $A_\lambda B_\lambda = C_\lambda D_\lambda$.
Let $a_\lambda$ and $b_\lambda$ be the images in $\mathcal L / \mathcal L^1$ of $A_\lambda$ and $B_\lambda$, respectively.
Then $\tau(a_\lambda, b_\lambda) = d(a_\lambda, b_\lambda)$, and the latter is norm continuous.  Thus we immediately obtain the following result.

\begin{corollary}
The function $\tau(a_\lambda, b_\lambda)$ is norm continuous.
\end{corollary}

Recall, however, that $\tau(a_\lambda, b_\lambda)$ is defined in terms of finite dimensional homology spaces $H_\bullet(A_\lambda)$, $H_\bullet(B_\lambda)$, $H_\bullet(C_\lambda)$, $H_\bullet(D_\lambda)$, and $H_\bullet(A_\lambda, B_\lambda, C_\lambda, D_\lambda)$.  In particular, the dimensions of these spaces are by no means continuous.  Thus, the continuity of joint torsion can be seen as analogous to the continuity of the Fredholm index.

Now let us record a number of properties of the determinant invariant consequent to Theorem \ref{thm:jt=di}.  Although the determinant invariant is defined as an infinite dimensional Fredholm determinant, Theorem \ref{thm:jt=di} shows that for commuting operators, it can actually be calculated in terms of finite dimensional data.  In fact, it can be computed as a determinant on a finite dimensional space.
Indeed, let $a$ and $b$ be commuting units in $\mathcal L / \mathcal L^1$ with commuting lifts $A$ and $B$.  Let $\mathcal E_{A+}$ be the direct sum of even terms in the exact sequence $\mathcal E_A$, and similarly for $\mathcal E_{A-}$.  Then $\mathcal E_{A+} \oplus \mathcal E_{A-}$ is a $\mathbf Z_2$-graded vector space with differentials $D_{A+}: \mathcal E_{A+} \to \mathcal E_{A-}$ and $D_{A-}: \mathcal E_{A-} \to \mathcal E_{A+}$.  For any algebraic pseudoinverse $D_{A-}^\dagger$ of $D_{A-}$, we can calculate $\tau(\mathcal E_A)$ in terms of the isomorphism $D_{A+} + D_{A-}^\dagger$ of finite dimensional spaces.  By Definition 3.3.2 and Equation 3.4 of \cite{Kaad} we have the following result.

\begin{corollary} \label{corollary:jt-finite-dimensional}
If $a$ and $b$ are units in $\mathcal L / \mathcal L^1$ with commuting lifts $A$ and $B$, then
\[
d(a,b) = (-1)^{\mu(A) + \mu(B)} \det (D_{B+} + D_{B-}^\dagger)^{-1} (D_{A+} + D_{A-}^\dagger)
\]
where $\mu(A) = \dim \ker A \cdot \dim \coker \, A$.  In particular, the determinant invariant $d(a,b)$ can be computed in terms of finite dimensional data.
\end{corollary}

In the case when the Koszul complex $K_\bullet(A,B)$ is acyclic, we have the following consequence of Theorem \ref{thm:jt=di} and Lemma \ref{lemma:mult-lefschetz}, which was first obtained in \cite[Theorem 1.1]{Reciprocity}.

\begin{theorem}[Carey-Pincus] \label{thm:jt-finite-dimensional2}
If $A$ and $B$ are commuting Fredholm operators and the Koszul complex $K_\bullet(A,B)$ is acyclic, we have
\[
d(a,b) = \frac{\det B|_{\ker A}}{\det B|_{\coker \, A}} \, \frac{\det A|_{\coker \, B}}{\det A|_{\ker B}}.
\]
\end{theorem}

There are analogues of Corollary \ref{corollary:jt-finite-dimensional} and Theorem \ref{thm:jt-finite-dimensional2} more generally in the case of almost commuting operators.  
Let $a$ and $b$ be commuting units in $\mathcal L / \mathcal L^1$.  Let $A, D \in \mathcal L$ be lifts of $a$, and let $B, C \in \mathcal L$ be lifts of $b$ such that $AB=CD$.
For simplicity, assume that the Koszul complex $K_\bullet(A,B,C,D)$ is acyclic.  Pick Fredholm operators $Q$ and $R$ such that $\ind \, Q = -\ind \, A$ and $\ind \, R = -\ind \, B$.  Then 
$\tilde A = A \oplus Q \oplus I$, 
$\tilde B = B \oplus I \oplus R$, 
$\tilde C = C \oplus I \oplus R$, and 
$\tilde D = D \oplus Q \oplus I$ 
all have index zero.  Moreover, the new Koszul complex $K_\bullet(\tilde A, \tilde B, \tilde C, \tilde D)$ is acyclic, and $d(A, B) = d(\tilde A, \tilde B) = \tau(\tilde A, \tilde B)$.
Then we find
\[
\tau(\mathcal E_{\tilde A, \tilde D}) = \tau( \tilde B: \ker \tilde D \to \ker \tilde A) \otimes \tau( \tilde C: \coker \, \tilde D \to \coker \, \tilde A)
\]
\[
\tau(\mathcal E_{\tilde B, \tilde C}) = \tau( \tilde D: \ker \tilde B \to \ker \tilde C) \otimes \tau( \tilde A: \coker \, \tilde B \to \coker \, \tilde C)
\]

For $T = \tilde A, \tilde B, \tilde C, \tilde D$, pick trace class operators $L_T$ and let $\pi_T$ be the quotient map as in Definition \ref{def:perturbation-vectors}.  Thus we have isomorphisms
$\pi_T L_T : \ker T \to \coker \, T$, 
and we calculate
\begin{align*}
\sigma_{\tilde A, \tilde D} & = \det (\tilde A+L_{\tilde A})^{-1} (\tilde D+L_{\tilde D}) \tau(\pi_{\tilde A} L_{\tilde A}) \otimes \tau (\pi_{\tilde D} L_{\tilde D})^* \\
\sigma_{\tilde B, \tilde C} & = \det (\tilde B+L_{\tilde B})^{-1} (\tilde C+L_{\tilde C}) \tau(\pi_{\tilde B} L_{\tilde C}) \otimes \tau (\pi_{\tilde B} L_{\tilde C})^*
\end{align*}

\begin{corollary}
In the situation above, we find that $d(a,b)$ is given by
\begin{multline*}
\frac{\det (\pi_{\tilde D} L_{\tilde D} )^{-1} (\tilde C|_{\coker \, \tilde D} )^{-1} (\pi_{\tilde A} L_{\tilde A} ) (\tilde B|_{\ker \tilde D} ) }{\det (\pi_{\tilde B} L_{\tilde B} )^{-1} (\tilde A|_{\coker \, \tilde B} )^{-1} (\pi_{\tilde C} L_{\tilde C} ) (\tilde D|_{\ker \tilde B} ) } \times \\
\qquad \qquad \times \det (\tilde A+L_{\tilde A})^{-1} (\tilde D+L_{\tilde D}) (\tilde B+L_{\tilde B})^{-1} (\tilde C+L_{\tilde C}).
\end{multline*}
\end{corollary}

\printbibliography

\end{document}